\documentclass[reqno]{amsproc}

\usepackage{amssymb, amsmath,amsthm,color}
\usepackage{url,hyperref}
\hypersetup{citecolor=blue, linkcolor=blue, colorlinks=true}
\usepackage[marginparwidth=1.5cm,margin=2cm]{geometry}

\theoremstyle{plain}
\newtheorem{theorem}{Theorem}[section]
\newtheorem{lemma}[theorem]{Lemma}
\newtheorem{proposition}[theorem]{Proposition}

\newcommand{\PG}{\mathrm{PG}}

\newcommand{\Tr}{\mathrm{Tr}}

\newcommand\s{\mathcal{S}}

\newcommand\F{\mathbb{F}}
\newcommand\h{\mathsf{H}}

\newcommand\PGammaU{\mathrm{P\Gamma U}}

\title{Every special set of the Hermitian surface $\h(3,q^2)$ is classical}

\author{John Bamberg}
\author{Ethan Kealley}
\address{Department of Mathematics and Statistics, 
The University of Western Australia, 35 Stirling Highway, Perth, W. A. 6009, Australia.}

\date{}

\keywords{Hermitian Veronesean, special set}

\subjclass[2000]{51E12, 05B25, 51E20}

\begin{document}

\begin{abstract}
Special sets of the Hermitian surface $\h(3,q^2)$, $q$ odd,
were introduced by Shult and Thas (1995) in order to construct new finite generalised quadrangles,
yet only one example is known to exist and it gives rise to a classical generalised quadrangle.
We show that there can be no other special sets of the Hermitian surface.
\end{abstract}

\maketitle

\section{Introduction}

Shult and Thas \cite[Section 6]{Shult:1995aa} introduced a method to construct finite generalised quadrangles of
order $(q^2,q^2)$ from certain sets of points of the Hermitian surface $\h(3,q^2)$. 
They were coined \emph{special sets} by Shult \cite{Shult2005}, and he showed that they
have many more connections to other objects in finite geometry than originally anticipated,
such as indicator sets, pseudo-ovals, and locally Hermitian ovoids.
A \emph{special set} of the Hermitian surface $\h(3,q^2)$, $q$ odd, is a set $\mathcal{S}$
of $q^2+1$ singular pairwise non-collinear points such that for any singular point $P$ not in $\mathcal{S}$, there
are 0 or 2 points of $\mathcal{S}$ collinear with $P$. Equivalently, a special set of $\h(3,q^2)$ is a set of $q^2+1$ points such that any three span a nondegenerate plane.
The only known examples are equivalent under the collineation group of $\h(3,q^2)$ to the \emph{Hermitian Veronesean} $\mathcal{V}$:
\[
\mathcal{V}:=\{ (1,x,x^q,x^{q+1}):x\in \F_{q^2}\}\cup\{(0,0,0,1)\}.
\]
Here, we have taken $\h(3,q^2)$ to be defined by the form 
\[
    h(X,Y):=X_0Y_3^q+X_3Y_0^q-X_1Y_1^q-X_2Y_2^q.
\]
We will refer to $\mathcal{V}$ as the \emph{classical} special set.
It has been a long-standing open problem in finite geometry whether or not every special set of $\h(3,q^2)$ is classical.
It was conjectured in \cite[Conjecture 2.1]{CossidenteSiciliano} that every special set is classical, 
and it was listed as `Open problem 6' in the recent survey article \cite{HirschfeldThas}.
Special sets were also studied in \cite{CMP2006, CossidentePenttila}.

There have been characterisations of the classical special set that use the notion of 
three elements (a \emph{triangle}) being \emph{in perspective}. For instance, a special set of $\h(3,q^2)$ such
that all triangles are in perspective is classical (see \cite[Corollary 5.3]{BambergMonzilloSiciliano}, \cite[Theorem 3.1]{Cossidente:2006aa} or \cite[Theorem 6.4]{Thas:2011aa}). If we fix a point $P$ of a special set $\s$, then having any triangle containing $P$ in perspective is 
enough to guarantee that $\s$ is classical \cite[Corollary 2.5]{BambergVandeVoorde}.
Recently, Bamberg and Van de Voorde \cite{BambergVandeVoorde} showed 
that if $\s$ is a special set and there are four non-coplanar points $P,Q_1,Q_2,Q_3\in \s$, such that all triangles $PQ_iR$, $i=1,2,3$, $R\in \s\setminus \{P,Q_i\}$, are in perspective, then $\s$ is 
classical. They also showed that if $\s$ satisfies a condition on Baer sublines of a tangent line,
and a weaker condition on triangles in perspective (namely that there is a point $Q$ such that for every $R\neq P,Q$ in $\s$, the triangle $PQR$ is in perspective), then $\s$ is classical.

In this paper, we give a complete resolution to the problem:

\begin{theorem}\label{maintheorem}
    Let $\mathcal{S}$ be a special set of $\h(3,q^2)$, $q$ odd. Then
    $\mathcal{S}$ is classical; that is, projectively equivalent to the Hermitian Veronesean.
\end{theorem}

We mention a by-product of Theorem \ref{maintheorem}. The Hermitian Veronesean, as a curve over the algebraic closure $\overline{\F}_{q^2}$,
is (up to projective equivalence) the unique non-planar rational curve of degree $q+1$ (see \cite{Ojiro}). Moreover,
the $\F_{q^2}$-maximal curves, those that attain the Hasse-Weil bound, are (up to birational isomorphism) the curves of degree 
$q+1$ embedded on a smooth Hermitian variety (see \cite{KorchmarosTorres}). In this paper, we have a characterisation of the $\mathcal{V}$ as a special set. See also \cite{LavrauwLiaPavese}.

\section{Setup and proof of the main result}

We introduce the relative trace and norm maps. First, the trace map $\Tr$ from $\F_{q^2}$ to $\F_q$ defined by $\Tr(x):=x+x^q$. The norm map $\mathrm{N}$ also
maps $\F_{q^2}$ to $\F_q$, and it is defined by $\mathrm{N}(x):=x^{q+1}$.
We will also need the following elementary fact about finite fields, which follows from Vieta's formula (since every $x\in F\backslash\{0\}$ is a root of $X^{q^2-1}-1$), but we give the details for completeness.

\begin{lemma} \label{product}
    For $F$ a finite field, we have that $\prod_{x\in F\backslash\{0\}}x=-1$.
\end{lemma}

\begin{proof}
    Since $F$ is a field, every non-zero element has a unique multiplicative inverse. Thus, if we multiply all non-zero elements, all elements will cancel except those that are self inverse, that is, those satisfying $x^2=1$. The only elements that satisfy this in any finite field are $1,-1$, so we have that 
    \[
        \prod_{x\in F\backslash\{0\}}x=1 \times -1=-1
    \]
    as required.
\end{proof}

Let $\mathcal{S}$ be a special set. Since the collineation group $\mathrm{P\Gamma U}(4,q)$ of $\h(3,q^2)$ acts transitively on pairs of non-collinear points, without loss of generality we can assume the points $P=(0,0,0,1)$ and $Q=(1,0,0,0)$ are in $\mathcal{S}$. 
Let $\ell$ be the totally isotropic line through the points $P$ and $X=(0,1,\omega,0)$, where $\mathrm{N}(\omega)=-1$. Then the points on $\ell$ other than $P$ have the form $(0,1,\omega,2t^q)$ for $t\in \mathbb{F}_{q^2}$. By \cite[Corollary 3.3]{BambergVandeVoorde}, the points of $\mathcal{S}$ other than $P$ have the form:
\begin{equation}
    R_t=\left(1,f(t)+t,(f(t)-t)\omega,c_t\right)
\end{equation}
where $f:\F_{q^2}\to \F_{q^2}$ is some function, and $\Tr(c_t)=2\Tr(f(t)t^q)$. We have that $R_0=Q$ as $Q$ and $X$ are collinear, and so we have that $f(0)=0$ and $c_0=0$ \cite[Corollary 3.4]{BambergVandeVoorde}. We also define $R_\infty:=P$ so that every point in $\mathcal{S}$ can be represented as $R_t$ for $t\in\mathbb{F}_{q^2}\cup\{\infty\}$. We will also let $h_{uv}:=h(R_u,R_v)$ for $u,v\in\mathbb{F}_{q^2}\cup\{\infty\}$. 

\begin{proposition}\label{product_is_one}
    For each $u\in\mathbb{F}_{q^2}\cup\{\infty\}$, we have that $\prod_{x\ne u}h_{ux}=1$.
\end{proposition}

\begin{proof}
    If $u=\infty$, then $h_{ux}=h(P,R_x)=1$, and so the statement is true. Thus, consider $u\in\mathbb{F}_{q^2}$. If $x=\infty$, then $h_{ux}=1$, so we are only interested in the case $x\in\mathbb{F}_{q^2}\backslash \{u\}$. Then, the point on $\ell$ collinear with $R_u$ is $Y_u=(0,1,\omega,2u^q)$. Consider the totally isotropic line $\ell_u:=\langle Y_u,R_u\rangle$. 
    
    Since $\s$ is a special set, there is a unique point on $\ell_u$ collinear with $R_x$ for each $x\in\mathbb{F}_{q^2}\backslash\{u\}$. Let this point be $Y_u+z_xR_u$. The $z_x$ are nonzero and distinct as we shall verify.
    First, $z_x\ne 0$. If $z_x=0$, then the corresponding point of $\langle Y_u,R_u\rangle$ is $Y_u$, so $Y_u$ is collinear with $R_x$. But $Y_u$ is already collinear with $P=R_\infty$ and $R_u$. Hence the three points $P, R_u, R_x$ all lie in the tangent plane $Y_u^\perp$. They therefore span a degenerate plane, contradicting the assumption that $\s$ is special. Second, the values $z_x$ are pairwise distinct. Suppose that $z_x=z_y$ for some distinct $x,y\in\F_{q^2}\setminus\{u\}$. Then the same point $Z\in \ell_u$ is collinear with both $R_x$ and $R_y$. Since $Z\in \ell_u$, it is also collinear with $R_u$. Thus $R_u, R_x, R_y$ all lie in the tangent plane $Z^\perp$, and hence span a degenerate plane. This again contradicts that $\s$ is a special set. Therefore the $q^2-1$ values $z_x$, $x\in\F_{q^2}\setminus\{u\}$ are $q^2-1$ distinct nonzero elements of $\F_{q^2}$. 
    
    Since $Y_u+z_xR_u$ and $R_x$ are collinear, we have that:
    \begin{align*}
        0&=h(Y_u+z_xR_u,R_x) = h(Y_u,R_x)+z_xh_{ux} 
%        \implies h_{ux}&=-\frac{h(Y_u,R_x)}{z_x}.
    \end{align*}
    and hence 
    \[
    h_{ux}=-\frac{h(Y_u,R_x)}{z_x}.
    \]
    We can then explicitly calculate $h(Y_u,R_x)$:
    \begin{align*}
        h(Y_u,R_x)&=2u^q-(f(x)+x)^q-\mathrm{N}(\omega)(f(x)-x)^q \\
        &=2u^q-2x^q \\
        &=2(u-x)^q.
    \end{align*}
    Then, the product of all the $h_{ux}$ is
    \begin{align*}
        \prod_{u\ne x}h_{ux}&=\prod_{u\ne x}\frac{2(x-u)^q}{z_x}.
    \end{align*}
    As mentioned earlier, the set of values of $z_x$ is equal to $\mathbb{F}_{q^2}^*$, and since $x\ne u$, we have that $2(x-u)^q$ is a bijection into $\mathbb{F}_{q^2}^*$. Thus, the fraction will cancel, and we get $\prod_{u\ne x}h_{ux}=1$ as required.
\end{proof}
% \john{All we need is for $h(Y_u,R_x)$ to be a bijection (in the variable $x$) into $\F_{q^2}^*$, and then we don't need the notational overhead. This is the only place that the shape of the $R_t$ is used. But, we would essentially be proving \cite[Corollary 3.3]{BambergVandeVoorde} anyway, so yeh, let's keep it as is.}

Thas \cite{Thas:2011aa} proved that a set of $q^2+1$ points $\s$ in $\h(3,q^2)$ such that all triangles are in perspective is equivalent to the Hermitian Veronesean. Indeed, we only need all triangles on a fixed point $P$
of $\s$ to be in perspective for the same result \cite[Corollary 2.5]{BambergVandeVoorde}.
Thus, in order to prove Theorem \ref{maintheorem}, we need to show that all triangles through a point $P$ of a special set are in perspective. We will make use of a result by Bamberg, Monzillo and Siciliano \cite{BambergMonzilloSiciliano} that three non-collinear points $A,B,C$ of $\h(3,q^2)$ are in perspective if and only if $[A,B,C]:=h(A,B)h(B,C)h(C,A)\in\F_q$. The quantity $[A,B,C]$ is known as the \emph{Segre-invariant} of the ordered triple of points $(A,B,C)$.

\begin{theorem}\label{inFq}
    For all $u,v\in\F_{q^2}\cup\{\infty\}$, we have $h_{uv}=h_{vu}$, and hence $h_{uv}\in \F_q$. 
\end{theorem}

\begin{proof}
    If $u=\infty$ or $v=\infty$, then $h_{uv}=h_{vu}=1$ as required. Thus, consider $u,v\in \F_{q^2}$, and $R_u,R_v\in\mathcal{S}$. The polar image $R_u^\perp$ of $R_u$ meets the polar image $R_v^\perp$ of $R_v$ in 
    a line intersecting $\h(3,q^2)$ in a Baer subline -- $q+1$ points. Let $Y$ be one of these points (which necessarily cannot lie in $\s$). 
    Consider the totally isotropic line $\langle Y,R_u\rangle$. Then the points on this line are of the form $Y+zR_u$ for $z\in\F_{q^2}\cup\{\infty\}$. By the $\{0,2\}$-property of special sets, all points on $\langle Y,R_u\rangle\backslash\{R_u\}$ are collinear to a unique point on $\mathcal{S}$. Take $R_x\in\mathcal{S}\backslash\{R_u,R_v\}$, then $R_x$ is collinear with $Y+z_xR_u$ for some $z_x\in\F_{q^2}^*$ as $z_x=0$ implies $R_x=R_v$, and $z_x=\infty$ implies $R_x=R_u$. Then $h(Y+z_xR_u,R_x)=0$, which when rearranged,
    yields $z_x=-\frac{h(Y,R_x)}{h_{ux}}$. Therefore,
    \begin{align*}
    \prod_{x\ne u,v} z_x &= \prod_{x\ne u,v}\left(-\frac{h(Y,R_x)}{h_{ux}}\right).
    \end{align*}
    Since $z_x$ ranges over $\F_{q^2}^*$ (see the proof of Proposition \ref{product_is_one}), by Lemma \ref{product}, we get:
    \begin{align*}
        -1&=h_{uv}\prod_{x\ne u,v}(-h(Y,R_x))\times\prod_{x\ne u}\frac{1}{h_{ux}},
    \end{align*}
    and so by Proposition \ref{product_is_one},
    \[        -\frac{1}{h_{uv}}=\prod_{x\ne u,v}h(Y,R_x).\]
    Note that the negative sign in the product disappears as there are $q^2-1$ factors --- an even number. Then, repeating the argument for the line $\langle Y,R_v\rangle$, we get that 
    \begin{equation*}
        -\frac{1}{h_{vu}}=\prod_{x\ne u,v}h(Y,R_x).
    \end{equation*}
    Thus, $h_{uv}=h_{vu}$ as required. Since $h$ is a Hermitian form, $h_{uv}\in \F_q$.
\end{proof}

Our main result is a corollary of Theorem \ref{inFq}:

\begin{proof}[Proof of Theorem \ref{maintheorem}]
    Let $\s$ be a special set. First, $\PGammaU(4,q)$ is transitive on pairs of non-collinear points (it is a rank 3 group),
    and so without loss of generality, we can assume that $P=(0,0,0,1)$ and $Q=(1,0,0,0)$ are in $\s$. 
    Write all the points of $\s$ as $R_t$ where $t\in \F_{q^2}\cup\{\infty\}$: where $R_0=Q$ and $R_\infty=P$.
    All we need to know is that the first homogeneous coordinate of each $R_t$ ($t\ne \infty$) is equal to 1,
    because then $h(P,R_t)=h(R_t,P)=1$.
    % We take a totally isotropic line $\ell:=\langle P, X\rangle$ where $X=(0,1,\omega,0)$ and $\mathrm{N}(\omega)=-1$. 
    % All
    % the points of $\s$ are of the form $R_t=(1,f(t)+t,(f(t)-t)\omega,c_t)$, and we 
    Let $h_{uv}:=h(R_u,R_v)$ for $u,v\in\mathbb{F}_{q^2}\cup\{\infty\}$. By Theorem \ref{inFq}, we have
    \begin{align*}
    [P,R_u,R_v]&=h(P,R_u)h(R_u,R_v)h(R_v,P)\\
    &=1\cdot h_{uv}\cdot 1 = h_{uv}\\
    &\in \F_q.
    \end{align*}
    By \cite[Proposition 5.1]{BambergMonzilloSiciliano}, $PR_uR_v$ is in perspective. Since this holds for all triangles on $P$,
    we have that $\s$ is classical by \cite[Corollary 2.5]{BambergVandeVoorde}.
\end{proof}

\section{Some remarks on the proof}

The proof was inspired by Segre's proof that every oval of $\PG(2,q)$, $q$ odd, is a conic.
We have already encountered terms that were inspired by Segre's theory.
To begin, the classical special set arises from the field reduction of a conic. Let $\mathcal{C}$ be
a conic defined by a non-singular quadratic form $Q$ on $\PG(2,q^2)$. Then 
field reduction to $\PG(5,q)$ maps the conic to $q^2+1$ lines $\mathcal{L}$ of an elliptic quadric $\mathsf{Q}^-(5,q)$.
Moreover, $\mathcal{L}$ consists of pairwise disjoint lines, any three of which are independent. Alternatively, 
every line that does not lie in $\mathcal{L}$ is concurrent with 0 or 2 elements of $\mathcal{L}$.
The Klein correspondence yields an isomorphism of the generalised quadrangle arising from the singular
points and lines of $\mathsf{Q}^-(5,q)$ with the generalised quadrangle $\h(3,q^2)$.
This isomorphism can be thought of as a duality of incidence structures, so $\mathcal{L}$
is mapped to $q^2+1$ points $\s$ such that every point that does not lie in $\s$ is collinear with 0 or 2 elements of $\s$.
The ternary property of three points $P, Q, R$ of $\s$ being in perspective, is the same concept
as that of the three corresponding points $\hat{P}, \hat{Q}, \hat{R}$ on the conic $\mathcal{C}$: the triangle
formed by the tangent lines at $\hat{P}, \hat{Q}, \hat{R}$, is in perspective from the triangle 
formed by $\hat{P}, \hat{Q}, \hat{R}$. The first step in Segre's theory is to prove that any 3-subset of
points of an oval of $\PG(2,q)$, $q$ odd, is in perspective. Part of that proof uses the clever observation (see Lemma \ref{product}) that the product of the nonzero elements of $\F_q$ is $-1$. The idea to take the product of the $h_{uv}$
in Proposition \ref{product_is_one} stems from M\"uller's proof of Segre's Theorem \cite{Muller}.

\section*{Declaration of generative AI use}

We used ChatGPT-5.5 to check our arguments and it found a simplification of the proof of Theorem \ref{inFq}. 
An earlier argument (by the authors) showed that the function $f$ in the expression of the $R_t$ was $\F_q$-linear (which also
proves Theorem \ref{maintheorem}), and ChatGPT observed that our expression for $h_{uv}$ is also valid for $h_{vu}$ 
and yields a shorter proof.

\section*{Acknowledgements}

The authors are extremely grateful to Geertrui Van de Voorde (The University of Canterbury) for reading a draft of our paper and
giving feedback.

\end{document}